\documentclass[12pt]{article}
\usepackage[a4paper, margin=2.5cm]{geometry}
\usepackage{amssymb,amsthm, amsmath}
\usepackage[hyphens]{url}
\usepackage{hyperref}

\newtheorem{theorem}{Theorem}[section]
\newtheorem{corollary}[theorem]{Corollary}
\newtheorem{lemma}[theorem]{Lemma}

\newtheorem{problem}[theorem]{Problem}
\newtheorem{conjecture}[theorem]{Conjecture}
\theoremstyle{definition}

\newcommand{\C}{\mathbb C}
\newcommand{\R}{\mathbb R}
\newcommand{\F}{\mathbb F}
\newcommand{\Q}{\mathbb Q}
\newcommand{\Z}{\mathbb Z}
\newcommand{\pts}{\mathcal P}
\newcommand{\cvs}{\mathcal C}
\newcommand{\cv}{C}
\newcommand{\eps}{\varepsilon}

\newcommand{\arxiv}[1]{\href{http://arxiv.org/abs/#1}{\tt arXiv:#1}}

\begin{document}
\title{A survey of Elekes-R\'onyai-type problems}
\author{Frank de Zeeuw
}
\date{}
\maketitle


\begin{abstract}
We give an overview of recent progress around a problem introduced by Elekes and R\'onyai.
The prototype problem is to show that a polynomial $f\in \R[x,y]$ has a large image on a Cartesian product $A\times B\subset \R^2$, unless $f$ has a group-related special form.
We discuss this problem and a number of variants and generalizations. 
This includes the Elekes-Szab\'o problem, which generalizes the Elekes-R\'onyai problem to a question about an upper bound on the intersection of an algebraic surface with a Cartesian product, 
and curve variants, where we ask the same questions for Cartesian products of finite subsets of algebraic curves.

These problems lie at the crossroads of combinatorics, algebra, and geometry:
They ask combinatorial questions about algebraic objects, whose answers turn out to have applications to geometric questions involving basic objects like distances, lines, and circles, as well as to sum-product-type questions from additive combinatorics.
As part of a recent surge of algebraic techniques in combinatorial geometry, a number of quantitative and qualitative steps have been made within this framework.
Nevertheless, many tantalizing open questions remain.
\end{abstract}

\tableofcontents

\section{The Elekes-R\'onyai problem}

\subsection{Sums, products, and expanding polynomials}
\label{subsec:exppoly}

Erd\H os and Szemer\'edi \cite{ES83} introduced the following problem in 1983.
Given a finite set $A$ in some ring, is it true that the \emph{sumset} $A+A$ or the \emph{productset} $A\cdot A$ must be large?
The rationale is that for an arithmetic progression the sumset is small, but the productset is large, while for a geometric progession, the reverse is true.
Erd\H os and Szemer\'edi proved for $A\subset \Z$ that
\begin{align}\label{eq:sumprod}
\max\{|A+A|,|A\cdot A|\} =\Omega\left(|A|^{1+c}\right) 
\end{align}
for a very small $c>0$.
This statement was later generalized to $\R$, and the constant has over the years been improved to $4/3+c'$ for a small $c'>0$ \cite{So09, KS15}.

The intuition behind this statement is that a set cannot have many ``coincidences'' for both addition and multiplication.
A statement like \eqref{eq:sumprod} is not the only way to capture this intuition.
Elekes \cite{E98} suggested that, since polynomials combine addition and multiplication, for most polynomials $f\in \R[x,y]$ we should have
\begin{align}\label{eq:exppoly1} 
|f(A\times A)| =\Omega\left(|A|^{1+c}\right),
\end{align}
for any finite $A\subset \R$, with a constant $c>0$ that may depend on $f$.
More generally, we should have 
\begin{align}\label{eq:exppoly2} 
|f(A\times B)| =\Omega\left(n^{1+c}\right),
\end{align}
for $A,B\subset \R$ with $|A|=|B|=n$, and a similar bound when $A$ and $B$ have different sizes.

Of course, \eqref{eq:exppoly2} cannot hold for all polynomials, as it fails when $f(x,y)=x+y$ and $A$ and $B$ are arithmetic progressions with the same difference, or when $f(x,y)=xy$ and $A$ and $B$ are geometric progressions with the same ratio. 
More generally, if $f$ has the additive form 
\begin{align}\label{eq:additive}
f(x,y) = g(h(x)+k(y))
\end{align}
with univariate polynomials $g,h,k$, then one has $|f(A\times B)|=O(n)$ if one chooses $A$ and $B$ in such a way that $h(A)$ and $k(B)$ are arithmetic progressions.
Similarly, if $f$ has the multiplicative form
\begin{align}\label{eq:multiplicative}
f(x,y)=g(h(x)\cdot k(y)),
\end{align}
then $|f(A\times B)|=O(n)$ if $h(A)$ and $k(B)$ are geometric progressions.
We will call a polynomial \emph{additive} if it has the form in \eqref{eq:additive}, and \emph{multiplicative} if it has the form in \eqref{eq:multiplicative}.

Elekes \cite{E98} conjectured that the additive form \eqref{eq:additive} and the multiplicative form  \eqref{eq:multiplicative} are the only exceptions to the bound \eqref{eq:exppoly2}.
Elekes proved a weaker form of this statement in \cite{E98}, and he collaborated with R\'onyai \cite{ER00} to prove this conjecture in full.
We state an improved version of the result due to Raz, Sharir, and Solymosi \cite{RSS15a}.

\begin{theorem}[Elekes-R\'onyai, Raz-Sharir-Solymosi]\label{thm:ERRSS}
Let $f\in \R[x,y]$ be a polynomial of degree $d$ that is not additive or multiplicative.
Then for all $A,B\subset \R$ with $|A|=|B|=n$ we have
\[|f(A\times B)| =\Omega_d\left(n^{4/3}\right).\]
\end{theorem}

To be precise, the bound stated in \cite{ER00} was $|f(A\times B)|=\omega(n)$, but inspection of the proof leads to a bound of the form $|f(A\times B)|=\Omega(n^{1+c_d})$ with a constant $c_d>0$ depending on the degree of $f$.
Raz, Sharir, and Solymosi \cite{RSS15a} showed that this dependence on the degree of $f$ is not necessary, and moreover improved the constant significantly.
Their proof used a setup inspired by Sharir, Sheffer, and Solymosi \cite{SSS13}, which gave a similar improvement on a special case (see Theorem \ref{thm:ERSSS}).
The same bound $|f(A\times B)| = \Omega(n^{4/3})$ was obtained by Hegyv\'ari and Hennecart \cite[Proposition 8.3]{HH13} for the polynomial $f(x,y) = xy(x+y)$.\footnote{The result in \cite[Proposition 8.3]{HH13} has the same proof setup as \cite{SSS13}, but it appears to be somewhat isolated; it makes no reference to \cite{ER00}, and in turn is not referred to in \cite{SSS13} and \cite{RSS15a}. 
This may be because \cite{HH13} primarily concerns expansion bounds over finite fields.
The arXiv publication date of \cite{SSS13} is half a year before that of \cite{HH13}.}

The exceptional role played by the additive and multiplicative forms suggests that groups play a special role in this type of theorem.
We will touch on this in Section \ref{sec:ES}, but see Elekes and Szab\'o \cite[Subsection 1.2]{ES12} for further discussion.
Other (older) expositions of Theorem \ref{thm:ERRSS} can be found in Elekes's magnificent survey \cite{E02}, and in Matou\v sek's book \cite[notes to Section 4.1]{M02}.


\subsection{Extensions}
\label{subsec:ERextensions}

It was observed in \cite{SSZ13} that a bound like \eqref{eq:exppoly2} should also hold when $A$ and $B$ have different sizes, and this was proved in a weak sense.
In \cite{RSS15a} such an ``unbalanced'' form of Theorem \ref{thm:ERRSS} was proved:
If $f$ is not additive or multiplicative, and $A,B\subset\R$, then 
\begin{align}\label{eq:unbalanced}
|f(A\times B)| = \Omega_d\left(\min\left\{|A|^{2/3}|B|^{2/3}, |A|^2, |B|^2\right\}\right).
\end{align}

Another way in which Theorem \ref{thm:ERRSS} can be extended is by replacing polynomials by \emph{rational functions}.
This was indeed done in \cite{ER00}, but not in \cite{RSS15a}.
Here the exceptions include the same forms \eqref{eq:additive} and \eqref{eq:multiplicative} with $g,h,k$ rational functions, but surprisingly, a third special form shows up here, namely
\begin{align}\label{ratspec}
f(x,y) = g\left(\frac{k(x)+l(y)}{1-k(x)l(y)}\right)
\end{align}
with rational functions $g,k,l$.
It was pointed out in \cite{BT12} that over $\C$ this can be seen as a multiplicative form, 
because $g((k(x)+l(y))/(1-k(x)l(y))) = G(K(x)L(y))$ if we set $G(z)=(z-1)/(i(z+1))$, $K(x) = (1+ik(x))/(1-ik(x))$, and $L(y) = (1+il(y))/(1-il(y))$ (and if we do some tedious computation).

It remains an open problem to improve the bound $|f(A\times B)|=\omega(n)$ of \cite{ER00} for rational functions $f$.
For one special case, the rational function $f(x,y)=(x-y)^2/(1+y^2)$, the bound $|f(A\times A)|=\Omega(|A|^{4/3})$ was proved in \cite{SSVZ15}.
This was done in the context of the distinct distance problem for distances between points and lines; 
$f$ gives the distance between the point $(x,0)$ and the line spanned by the points $(y,0)$ and $(0,1)$.

\begin{problem}
Prove Theorem \ref{thm:ERRSS} for rational functions $f\in \R[x,y]$ that are not additive, multiplicative, or of the form \eqref{ratspec}.
\end{problem}

Yet another way to extend Theorem \ref{thm:ERRSS} is from $\R$ to $\C$.
Most of the proof in \cite{RSS15a} extends easily to $\C$, with the exception of the incidence bound used (see Subsection \ref{subsec:ERproof}), which can be replaced by the bound over $\C$ proved in \cite{SdZ15} (see Theorem \ref{thm:SZ}); the details are written down in \cite{Ze15}.
Thus Theorem \ref{thm:ERRSS} holds also over $\C$.

The exponent $4/3$ in Theorem \ref{thm:ERRSS} is most likely not optimal.
Elekes \cite{E99} in fact conjectured that the bound in Theorem \ref{thm:ERRSS} can be improved as far as $\Omega(n^{2-\eps})$,
but no exponent better than $4/3$ has been established for any polynomial.
Elekes \cite{E99} noted that for $f(x,y) = x^2+xy+y^2$ (and many other polynomials) and $A=B=\{1,\ldots,n\}$ we have  $|f(A\times B)| = \Theta(n^2/\sqrt{\log n})$ (see \cite[Chapter 6]{S15} for details), so perhaps the bound in Theorem \ref{thm:ERRSS} can even be improved to $\Omega(n^2/\sqrt{\log n})$.

Let us combine the above extensions to conjecture the ultimate Elekes-R\'onyai-type theorem (although see Section \ref{sec:variants} for further variants).

\begin{conjecture}
Let $f\in \C(x,y)$ be a rational function of degree\footnote{The maximum of the degrees of the numerator and denominator, assuming that these do not have a common factor.} $d$ that is not additive or multiplicative.
Then for all $A,B\subset \C$ with $|A|=|B|=n$ we have
\[|f(A\times B)| = \Omega_{d,\eps}\left(n^{2-\eps}\right).  \]
\end{conjecture}


\subsection{Applications}
\label{subsec:ERapps}

\paragraph{Sum-product bounds.}
For a first consequence of Theorem \ref{thm:ERRSS}, we return to the sum-product problem mentioned at the start of Subsection \ref{subsec:exppoly}.
The following generalization of the bound \eqref{eq:sumprod} was proved by Shen \cite{S12}: If $A\subset \R$ and $f\in \R[x,y]$ is a polynomial of degree $d$ that is not of the form $g(\ell(x,y))$, with $g$ a univariate polynomial and $\ell$ a linear bivariate polynomial, 
then 
\begin{align}\label{eq:shen}
\max\{|A+A|,|f(A\times A)|\} =\Omega_d\left(|A|^{5/4}\right). 
\end{align}
For many polynomials, Theorem \ref{thm:ERRSS} improves this bound, and it also shows that in those cases one does not need to consider $|A+A|$ to conclude that $|f(A\times A)|$ is large. 

On the other hand, there are many polynomials that have the special form of Theorem \ref{thm:ERRSS}, but that do not have the special form of Shen.
Even in those cases, it may be possible to obtain a bound on $|f(A\times A)|$ independent of $|A+A|$; for instance, Elekes, Nathanson, and Ruzsa \cite{ENR99} prove $|f(A\times A)| = \Omega_d(|A|^{5/4})$ for $f(x,y) = x+y^2$ (and many similar functions).
Note that for this bound it is crucial that the Cartesian product is of the form $A\times A$ rather than $A\times B$.
It may be that such a bound holds for any $f$ that is not of the form $g(h(x)+h(y))$ or $g(h(x)\cdot h(y))$; this question does not seem to have been studied.

\paragraph{Distances between lines.}
As a corollary of their result, Elekes and R\'onyai \cite{ER00} made progress on the following problem of Purdy (see \cite[Section 5.5]{BMP05}):
Given two lines with $n$ points each, what is the minimum number of distances occurring between the two point sets?
This problem is a simpler variant of the distinct distances problem of Erd\H os \cite{E46}, which asks for the minimum number of distinct distances determined by a point set in the plane.
Erd\H os's problem was almost completely solved by Guth and Katz \cite{GK15} using new algebraic methods.

In Purdy's problem there are two exceptional situations, when the two lines are parallel or orthogonal.
Indeed, if on two parallel lines one places two arithmetic progressions of size $n$ with the same common difference, then the number of distinct distances is linear in $n$.
On two orthogonal lines, say the $x$-axis and the $y$-axis, one can take the sets $\{(\sqrt{i},0):1\leq i\leq n\}$ and $\{(0,\sqrt{j}):1\leq j\leq n\}$ to get a linear number of distances.
The following theorem states that for all other pairs of lines there are considerably more distances.
Given $P_1,P_2\subset \R^2$, we write $D(P_1,P_2)$ for the set of Euclidean distances between the points of $P_1$ and the points of $P_2$.

\begin{theorem}[Elekes-R\'onyai, Sharir-Sheffer-Solymosi]\label{thm:ERSSS}
Let $L_1,L_2$ be two lines in $\R^2$ that are not parallel or orthogonal,
and let $P_1\subset L_1,P_2\subset L_2$ be finite sets of size $n$.
Then the number of distinct distances between $P_1$ and $P_2$ satisfies
\[|D(P_1,P_2)| = \Omega\left(n^{4/3}\right).\]
\end{theorem}
\begin{proof}[Proof sketch.]
We can assume that the lines are $y=0$ and $y=mx$, with $m\neq 0$.
The squared distance between $(s,0)$ and $(t,mt)$ is
\[f(s,t) = (s-t)^2+m^2t^2.
 \]
It is easy to verify that the polynomial $f(s,t)$ is not additive or multiplicative, 
so Theorem \ref{thm:ERRSS} implies the stated bound.
\end{proof}

Elekes and R\'onyai first proved a superlinear bound in the case $|P_1| = |P_2|$ as a consequence of their result in \cite{ER00}, thus solving Purdy's problem, in the qualitative sense of distinguishing the special pairs of lines.
Elekes \cite{E99} then\footnote{The chronology is somewhat confusing here. The paper \cite{ER00} was published in 2000, and \cite{E99} in 1999. However, \cite{E99} refers back to \cite{ER00} and makes it clear that \cite{E99} is an improvement on a special case of \cite{ER00}.} quantified the proof from \cite{ER00} in this special case to obtain a short proof of the explicit bound $\Omega(n^{5/4})$, noting \cite{E02} that Brass and Matou\v sek asked for such a ``gap theorem''.
An unbalanced form was proved in \cite{SSZ13}.
The bound in Theorem \ref{thm:ERSSS} was obtained by Sharir, Sheffer, and Solymosi \cite{SSS13}, using a proof inspired by that of Guth and Katz \cite{GK15}\footnote{A preprint version of \cite{GK15} became available in 2010.}. 

An earlier version of \cite{SSS13} used the Elekes-Sharir transformation from \cite{ES11} that was crucial in \cite{GK15} to connect distances with incidences; it was then observed that in Purdy's problem a considerably easier incidence problem can be obtained, and also that the Elekes-Sharir transformation can be bypassed. 
The result was a proof that is even simpler than that of \cite{E99}, and its simplicity allowed for many generalizations, including Theorem \ref{thm:ERRSS} and many other results in this survey.

\paragraph{Directions on curves.}
The distinct directions problem asks for the minimum number of distinct directions determined by a non-collinear point set in the plane.
It is superficially similar to the distinct distances problem, in the sense that it asks for the minimum number of distinct values of a function of pairs of points in the plane.
However, it was solved exactly by Ungar \cite{U82} in 1982: Any non-collinear set $P$ in $\R^2$ determines at least $|P|-1$ distinct directions.

This leaves the more difficult \emph{structural} question: What is the structure of sets that determine few distinct directions?
Let us write $S(P)$ for the set of directions (or slopes) determined by $P\subset \R^2$.
Elekes \cite{E99b} conjectured that if $|S(P)| = O(|P|)$, then $P$ must have many points on a conic; even in the weakest form, where ``many'' is six, this is unknown.
Elekes \cite{E99b} showed that a (very) restricted version of this conjecture follows from Theorem \ref{thm:ERRSS}: If $P$ lies on the graph of a polynomial of degree at most $d$ and has $|S(P)| = O_d(|P|)$, then the polynomial must be linear or quadratic.
See \cite[Subsection 3.3]{E02} for a detailed discussion.
We state here the improvement of this result from \cite{RSS15a}, obtained as a consequence of Theorem \ref{thm:ERRSS}.
In Subsection \ref{subsec:collineartriples} we will discuss the same question for points sets on arbitrary algebraic curves.

\begin{corollary}\label{cor:directionsongraphs}
Let $P$ be a finite point set that is contained in the graph $y=g(x)$ of a polynomial $g\in \R[x]$ of degree $d\geq 3$.
Then the number of distinct directions determined by $P$ satisfies
\[|S(P)| = \Omega_d\left(|P|^{4/3}\right). \]
\end{corollary}
\begin{proof}[Proof sketch.]
The direction determined by the points $(s,g(s)), (t,g(t))$ is given by the polynomial
\[f(s,t) = \frac{g(s)-g(t)}{s-t}. \]
It is not hard to verify that this polynomial is not additive or multiplicative (except when $g$ is linear or quadratic), so Theorem \ref{thm:ERRSS} gives the stated bound.
\end{proof}

When $g$ has degree less than three, the number of directions can be linear.
Take for instance the parabola $y=x^2$ and the point set 
$\{(i,i^2): 1\leq i\leq n\}$.
Then the determined directions are $(i^2-j^2)/(i-j)=i+j$, so there are $O(n)$ distinct directions.
It is not hard to see that a similar example can be constructed for any other quadratic $g$.

 \subsection{About the proof of Theorem \ref{thm:ERRSS}}\label{subsec:ERproof}

We now discuss the proof of Theorem \ref{thm:ERRSS}. 
We certainly do not give a full account of the proof in \cite{RSS15a}, but we introduce the setup without too much technical detail.
The proof of the weaker bound in \cite{ER00} used related techniques but was inherently different.
The proof setup in \cite{RSS15a} originated in  \cite{SSS13}, and seems to have been discovered independently (and somewhat later) in \cite[Proposition 8.3]{HH13}; a glimpse of this setup can also be seen in the (earlier) proof of \cite[Theorem 27]{ES12}.
We use the word ``setup'' to refer to the overall counting scheme of \cite{RSS15a}, which is only the surface layer of the proof.
The real achievement in \cite{RSS15a} was the treatment of high-multiplicity curves, which we discuss only briefly at the end of this subsection.

The main tool for obtaining the bound in Theorem \ref{thm:ERRSS} is incidence theory, specifically an incidence bound for points and curves of Pach and Sharir \cite{PS98}\footnote{The proof in \cite{RSS15a} did not directly use the bound from \cite{PS98}, but rather adapted the proof from \cite{PS98} to the incidence situation in \cite{RSS15a}.}. This incidence bound is a generalization of the classical theorem of Szemer\'edi and Trotter \cite{ST} that bounds incidences between points and lines.
The following version is particularly convenient for the applications in this survey. 
It assumes that the point set is a Cartesian product, 
which allows for a significantly simpler proof over $\R$, and makes it easier to prove an analogue over $\C$ (where the corresponding bound has not yet been established without extra assumptions; see \cite{SZ15}).
This theorem was proved by Solymosi and De Zeeuw \cite{SdZ15}, although the real case was probably folklore.
We write $|I(\pts,\cvs)|$ for the set of incidences between the points $\pts$ and the curves $\cvs$, i.e., the set of pairs $(p,C)\in \pts\times \cvs$ such that $p\in C$.
We give a quick sketch of the proof, to show that proving this tool does not require any heavy machinery (at least over $\R$).

\begin{theorem}\label{thm:SZ}
Let $\pts=A\times B$ be a Cartesian product in $\R^2$ or $\C^2$, and let $\cvs$ be a set of algebraic curves of degree at most $d$ in the same plane.
Assume that any two points of $\pts$ are contained in at most $M$ curves of $\cvs$.
Then
\[|I(\pts,\cvs)| = O_{d,M}\left(|\pts|^{2/3}|\cvs|^{2/3} +|\pts|+|\cvs|\right). \]
\end{theorem}
\begin{proof}[Proof sketch in $\R^2$.]
Let us assume that $|A|=|B|=n$ and $\cvs = n^2$ (which roughly holds in most of the statements in this survey).
We can partition $\R^2$ using $O(r)$ horizontal and vertical lines, in such a way that each of the $O(r^2)$ resulting rectangles contains roughly $O(n^2/r^2)$ points of $A\times B$.
Moreover, we can ensure that the partitioning lines do not contain any points of $A\times B$, and are not contained in any of the curves in $\cvs$.
We observe that an algebraic curve of degree at most $d$ intersects $O_d(r)$ rectangles, since it can intersect a partitioning line in at most $d$ points.

We split the incidences as follows: $I_1$ is the set of incidences $(p,C)$ such that $p$ is the only incidence on $C$ in the rectangle that $p$ lies in, and $I_2$ is the remaining set of incidences.
Since each curve hits $O_d(r)$ rectangles, we have $|I_1| = O_d(rn^2)$.
On the other hand, if a curve has an incidence from $I_2$ in a certain rectangle,
then it has at least one additional incidence in that same rectangle.
By the assumption of the theorem, any two points are together contained in at most $M$ curves. 
Thus the $O(n^2/r^2)$ points in a rectangle are involved in at most $O_M(n^4/r^4)$ incidences from $I_2$, which altogether gives $|I_2| = O_{M}(n^4/r^2)$.
Choosing $r =n^{2/3} $ optimizes 
\[|I(\pts,\cvs)| = |I_1|+|I_2| = O_{d,M}\left(rn^2 + n^4/r^2\right) 
= O\left(n^{8/3}\right),\]
which is the stated bound when $|\pts| = |\cvs| = n^2$.
\end{proof}

The proof of Theorem \ref{thm:ERRSS} is based on an upper bound for the size of the following set of quadruples:
\[Q=\{(a,b,a',b')\in A\times B\times A\times B: f(a,b) = f(a',b')\}.\]
Given such an upper bound, the Cauchy-Schwarz inequality gives a lower bound on the size of the image set $f(A\times B)$, using the following calculation:
\begin{align}\label{eq:quads}
|Q| = \sum_{c\in f(A\times B)}|f^{-1}(c)|^2
\geq \frac{1}{|f(A\times B)|} \left(\sum_{c\in f(A\times B)} |f^{-1}(c)|\right)^2
= \frac{n^4}{|f(A\times B)|}.
\end{align}
Specifically, when $f$ is not additive or multiplicative we obtain the upper bound $|Q|=O_d(n^{8/3})$, and then \eqref{eq:quads} implies $|f(A\times B)| =\Omega_d(n^{4/3})$.
A similar application of Cauchy-Schwarz played a central role in \cite{GK15}.

To obtain an upper bound on $|Q|$, we define a set of curves and a set of points based on the given polynomial $f$ and the given sets $A,B$, and then we apply Theorem \ref{thm:SZ}.
For each $(a,a')\in A\times A$, define
\[C_{aa'} = \{(x,y)\in \R^2: f(a,x) = f(a',y)\}. \]
This is an algebraic curve of degree at most $d$ (the degree of $f$).
Note that for $(b,b')\in B\times B$, we have $(b,b')\in C_{aa'}$ if and only if $(a,b,a',b')\in Q$.
Thus, if we set 
\[\pts = B\times B~~~\text{and}~~~\cvs =\{C_{aa'}: (a,a')\in A\times A\},\]
then $|I(\pts,\cvs)| = |Q|$.

If we could apply Theorem \ref{thm:SZ} to $\pts$ and $\cvs$, then we would immediately get the desired bound $|Q| = O_d(n^{8/3})$.
However, $\pts$ and $\cvs$ need not satisfy the degrees-of-freedom condition of Theorem \ref{thm:SZ} that two points are contained in a bounded number of curves.
This is to be expected, since the bound should fail when $f$ is additive or multiplicative.
In fact, even when $f$ is not additive or multiplicative, the degrees-of-freedom condition may be violated.
However, it was shown in \cite{RSS15a} that when $f$ is not additive or multiplicative, the condition is only violated in a weak sense.
Specifically, one can remove negligible subsets of the points and curves so that the remainder does satisfy the condition.

A key insight in the proof is that the curves $C_{aa'}$ satisfy a kind of \emph{duality}.
Indeed, we can define ``dual curves'' of the form $C^*_{bb'} = \{(s,t)\in \R^2: f(s,b) = f(t,b')\}$, so that the point $(a,a')$ lies on the dual curve $C^*_{bb'}$ if and only if the point $(b,b')$ lies on the curve $C_{aa'}$.
Thus, to check the degrees-of-freedom condition of Theorem \ref{thm:SZ} that two points $(b_1,b_1'),(b_2,b_2')$ lie on a bounded number of curves $C_{aa'}$, we can instead look at the number of points $(a,a')$ in the intersection of the curves $C^*_{b_1b_1'},C^*_{b_2b_2'}$.
Such an intersection is easily bounded by B\'ezout's inequality, unless the curves $C^*_{b_1b_1'}$ and $C^*_{b_2b_2'}$ have a common component.
Thus the degrees-of-freedom condition comes down to showing that when many of the curves $C^*_{bb'}$ have many common components, with high multiplicity, then $f$ must be additive or multiplicative.
By symmetry, we may as well consider this question for the original curves $C_{aa'}$.

Let us see what happens when $f$ is additive or multiplicative.
First consider the case where $f(x,y) = h(x) + k(y)$.
Then $C_{aa'}$ is defined by $k(x) - k(y) = h(a') - h(a)$.
Thus $C_{a_1a_1'}$ and $C_{a_2a_2'}$ are the same curve whenever $h(a_1') -h(a_1) = h(a_2') - h(a_2)$; this means that as many as $\Theta(|A|)$ pairs $(a,a')$ may define the same curve $C_{aa'}$.
Similarly, when $f(x,y) = h(x)k(y)$, then $C_{aa'}$ is defined by $k(x) = k(y)\cdot (h(a')/h(a))$, and again we can have high multiplicity.
Finally, when $f(x,y) = g(h(x)+k(y))$, then $f(a,x) - f(a',y)$ has the factor  $h(a) + k(x) -h(a')-h(y) $, which corresponds to a component that can have high multiplicity
(and the same happens for $f(x,y) = g(h(x)\cdot k(y))$).

The key challenge in the proof of Theorem \ref{thm:ERRSS} is to obtain the converse, i.e., to show that when many curves have high multiplicity, then there must be polynomials $g,h,k$ that explain the multiplicity in one of the ways above.
In \cite{RSS15a} this is done by algebraically prying out $g,h,k$ from specific coefficients of the polynomial $f$.
For instance, roughly speaking, when the curves $C_{aa'}$ have many common components and the coefficient of the leading term of $f(a,x)-f(a',y)$ is not constant as a polynomial in $a$, then this polynomial turns out to be the $h$ in the multiplicative form $f(x,y) = g(h(x)k(y))$. 
On the other hand, if only the constant term of $f(a,x)-f(a',y)$ depends on $a$ and $a'$, then this leads to the polynomial $h$ in the additive form $f(x,y) = g(h(x)+k(y))$.

\section{The Elekes-Szab\'o problem}
\label{sec:ES}

\subsection{Intersecting varieties with Cartesian products}
To introduce a generalization of the Elekes-R\'onyai problem due to Elekes and Szab\'o, we take a step back and approach from a different direction; after a while we will see what the connection between the problems is.
In this section we work primarily over $\C$, which is the most natural setting for the Elekes-Szab\'o problem and the relevant proofs.

Let us consider the Schwartz-Zippel lemma (see \cite{L09} for the curious history of this lemma).
The simplest non-trivial case is the following bound on the intersection of a curve with a Cartesian product\footnote{The word ``grid'' is often used in this context, but may lead to confusion with integer grids.}. 
If $F\in \C[x,y]$ is a polynomial of degree $d$ and $A,B\subset \C$ are finite sets of size $n$,
then\footnote{We write $Z(F)$ for the \emph{zero set} of a polynomial $F$, i.e., the set of points at which $F$ vanishes.}
\begin{align}\label{eq:SZC2}
|Z(F)\cap (A\times B)| = O_d\left(n\right). 
\end{align}
This statement is ``tight'' in the sense that, for any fixed polynomial, there are sets $A,B$ for which the bound is best possible.
Indeed, we can arbitrarily choose $n$ points on $Z(F)$, let $A$ be the projection of this set to the $x$-axis, and let $B$ be the projection to the $y$-axis; then $A\times B$ shares at least $n$ points with $Z(F)$.

Now let us consider the tightness for the next case of the Schwartz-Zippel lemma (see Subsection \ref{subsec:longer} for the general statement), which says that for $F\in \C[x,y,z]$ and $A,B,C\subset \C$ of size $n$ we have
\begin{align}\label{eq:SZ3D}
|Z(F)\cap (A\times B\times C)| = O_d\left(n^2\right).
\end{align}
It is not so clear if this bound is tight, 
since the trick used above to show that \eqref{eq:SZC2} is tight does not work here. 
The best we could try is to choose $A, B$ of size $n$, take $n^2$ points on $Z(F)$ above $A\times B$, and then project to the $z$-axis to get $C$; but the resulting $C$ is likely to have many more than $n$ points.

Nevertheless, for certain special polynomials the bound in \eqref{eq:SZ3D} is tight.
Take for instance $F = x+y-z$ and $A=B=C=\{1,\ldots,n\}$;
then $|Z(F)\cap (A\times B\times C)| = \Theta(n^2)$. 
Of course, one can construct similar examples for any polynomial of the form $F=f(g(x)+h(y)+k(z))$ with $f,g,h,k$ univariate polynomials.
In analogy with Theorem \ref{thm:ERRSS}, one might guess that these are the only special polynomials, but this is not quite true.
It turns out that the right class of special polynomials consists of those of the form $F=f(g(x)+h(y)+k(z))$ with $f,g,h,k$ \emph{analytic} functions that are defined in a \emph{local} way.

\begin{theorem}[Elekes-Szab\'o, Raz-Sharir-De Zeeuw]\label{thm:ESRSZ}
Let $F\in \C[x,y,z]$ be an irreducible polynomial of degree $d$ with 
with each of $F_x,F_y, F_z$ not identically zero. Then one of the following holds.\\
$(i)$ For all $A,B,C\subset \C$ with $|A|=|B|=|C|=n$ we have
$$|Z(F) \cap (A\times B\times C)|=O_d(n^{11/6}).$$
$(ii)$ There exists a one-dimensional subvariety $Z_0\subset Z(F)$, such that every $v\in Z(F)\backslash Z_0$ has an open neighborhood $D_1\times D_2\times D_3$ and analytic functions 
$\varphi_i: D_i\to \C$,
such that for every $(x,y,z)\in D_1\times D_2\times D_3$ we have
$$(x,y,z)\in Z(F)~~~\text{if and only if}~~~\varphi_1(x)+\varphi_2(y)+\varphi_3(z) = 0.$$
\end{theorem}

Qualitatively, this result was proved by Elekes and Szab\'o \cite{ES12}\footnote{The publication year of \cite{ES12} is 2012, but an essentially complete version of the paper existed much earlier; Elekes \cite{E02} referred to the result in 2002.}, 
who proved that $|Z(F)\cap(A\times B\times C)| = O(n^{2-\eta_d})$ for a constant $\eta_d$ depending only on the degree $d$ of $F$,\footnote{Earlier versions of \cite{ES12} claimed a bound of the form $O(n^{2-\eta})$ for an absolute $\eta>0$, and this was restated in \cite{E02}. But the published version states the theorem with $\eta_d$ depending on $d$.} 
unless $F$ has the special form described in Theorem \ref{thm:ESRSZ}$(ii)$ below.
Using the new proof setup in \cite{SSS13,RSS15a}, 
this bound was improved to $O(n^{11/6})$ by Raz, Sharir, and De Zeeuw \cite{RSZ15}.
An exposition of the algebraic geometry underlying the proof in \cite{ES12} was written by Wang \cite{W15}.

The statement also holds if $\C$ is replaced by $\R$ (and the analytic $\varphi_i:D\to \C$ are replaced by real-analytic $\varphi_i:D\to \R$).
We can also allow $A,B,C$ to have different sizes.
This does not affect the description in condition $(ii)$, and the bound in condition $(i)$ becomes
\begin{align}\label{eq:ESunbalanced}
|Z(F) \cap (A\times B\times C)|=O_d\left( |A|^{2/3}|B|^{2/3}|C|^{1/2}+|A||C|^{1/2}+|B||C|^{1/2}+|C|\right);
\end{align}
of course the same bound holds for any permutation of $A,B,C$.
We can again conjecture that the bound can be significantly improved, perhaps as far as $O_\eps(n^{1+\eps})$ (in the balanced case), or even $O(n\sqrt{\log n})$.
No better lower bound is known than $\Omega(n\sqrt{\log n})$, for instance for $F =  x^2+xy+y^2 - z$ (which comes directly from the polynomial $f = x^2+xy+y^2$ that provides the best known upper bound for Theorem \ref{thm:ERRSS}, as mentioned in Subsection \ref{subsec:ERextensions}).

It is not likely that condition $(ii)$ can be replaced by a purely polynomial condition, i.e., without mentioning analytic functions.
This can be seen from the fact that the group law on an elliptic curve gives constructions for which the bound in $(i)$ does not hold, while on the other hand, it is well-known that parametrizing elliptic curves as in $(ii)$ requires analytic functions.
We will give more details on the connection with elliptic curves at the end of Subsection \ref{subsec:ESoncurves}.

In \cite{ES12}, condition $(ii)$ of Theorem \ref{thm:ESRSZ} is formulated in a somewhat stronger ``global'' form, although for all applications in this survey, the formulation in Theorem \ref{thm:ESRSZ} seems to be more convenient.
Specifically, the local functions $\varphi_i$ can be replaced by \emph{analytic multi-functions} from $\C$ to a \emph{one-dimensional connected algebraic group} $\mathcal{G}$, so that $Z(F)$ is the image of the variety $\{(x,y,z)\in \mathcal{G}^3: x\oplus y\oplus z = e\}$; we refer to \cite{ES12, W15} for definitions.

\subsection{A derivative test for special \texorpdfstring{$F$}{F}}

In applications, it may not be easy to determine whether a given polynomial $F$ satisfies condition $(ii)$.
For relatively simple polynomials, we have the following derivative condition.
It is mentioned in \cite[Subsection 1.1]{ER00} and stated in \cite[Lemma 33]{ES12}; in \cite{ER00} the sufficiency of the condition is ascribed to Jarai, although no proof is provided in \cite{ER00} or \cite{ES12}.
We give a short sketch of the proof; a detailed proof can be found in \cite{RS15a}.

\begin{lemma}\label{lem:derivative}
Let $f:\R^2\to\R$ be a twice-differentiable function with $f_y\not\equiv 0$.
There exist differentiable functions $\psi,\varphi_1,\varphi_2:\R\to\R$ such that 
\begin{equation}\label{eq:testform}
f(x,y)= \psi(\varphi_1(x)+\varphi_2(y))
\end{equation}
if and only if
\begin{equation}
\frac{\partial^2(\log|f_x/f_y|)}{\partial x\partial y}\equiv 0.
\end{equation}
The same holds for analytic $f:\C^2\to \C$ with $\psi,\varphi_1,\varphi_2:\C\to \C$ analytic.
\end{lemma}
\begin{proof}[Proof sketch.]
If $f(x,y)= \psi(\varphi_1(x)+\varphi_2(y))$, then $f_x/f_y = \varphi_1'(x)/\varphi_2'(y)$, 
which gives $\log|f_x/f_y| = \log|\varphi_1'(x)| - \log|\varphi_2'(y)|$.
Differentiating with respect to $x$ and $y$ gives $0$.

Conversely, if $\partial^2(\log|f_x/f_y|)/\partial x\partial y\equiv 0$,
integrating gives $\log|f_x/f_y| = g_1(x) - g_2(y)$.
Then, setting $\varphi_1(x) = \int e^{g_1(x)}dx$ and 
$\varphi_2(y) = \int e^{g_2(y)}dy$,
we have $f_x/f_y = \varphi_1'(x)/\varphi_2'(y)$.
We express $f$ in terms of the new variables $u = \varphi_1(x)+\varphi_2(y)$ and $v = \varphi_1(x) - \varphi_2(y)$,
so that the chain rule gives $f_x = \varphi_1'(x)(f_u+f_v)$ and $f_y = \varphi_2'(y)(f_u-f_v)$.
Combining these equations gives $0 = \frac{f_x}{\varphi_1'(x)} - \frac{f_y}{\varphi_2'(y)} = 2f_v$.
Thus $f$ depends only on the variable $u$, which means that we can write it as $f(x,y) = \psi(\varphi_1(x)+\varphi_2(y))$.
\end{proof}

To apply this lemma to a polynomial $F(x,y,z)$, we need to locally write the implicit surface $Z(F)$ as an explicit surface $z = f(x,y)$, for an analytic function $f$. 
Then the expression $\varphi_1(x)+\varphi_2(y) + \varphi_3(z)=0$ in condition $(ii)$ of Theorem \ref{thm:ESRSZ} is equivalent to $f(x,y)= \varphi_3^{-1}(\varphi_1(x)+\varphi_2(y))$.
In theory, such an $f$ exists by the implicit function theorem, but in practice we can only calculate $f$ when $F$ has low degree (in one of the variables).

\subsection{Applications}

\paragraph{Expanding polynomials.}
Given $f\in \C[x,y]$, we can set $F(x,y,z) = f(x,y) - z$ and apply Theorem \ref{thm:ESRSZ} with $|A|=|B|=n$ and $C = f(A\times B)$.
If condition $(i)$ applies, we get from the unbalanced bound \eqref{eq:ESunbalanced} that
\[n^2=|Z(f(x,y)-z)\cap (A\times B\times C)|=O_d(n^{4/3}|C|^{1/2}),\]
so $|f(A\times B)| = |C| = \Omega(n^{4/3})$.
Otherwise, condition $(ii)$ tells us that locally we have
\begin{align}\label{eq:analyticspecial}
f(x,y) = \psi(\varphi_1(x)+\varphi_2(y)),
\end{align}
with $\psi, \varphi_1,\varphi_2$ analytic functions.

Note that the multiplicative form of $f$ also falls under \eqref{eq:analyticspecial}, since we can (locally) write 
\[g(h(x)\cdot k(y)) = (g\circ \log^{-1})(\log |h(x)| + \log |k(y)| )\]
with all functions analytic.
We thus have almost deduced Theorem \ref{thm:ERRSS}, except that the special form of $f$ is local, and we do not know that $\varphi_1,\varphi_2,\psi$ are polynomials.
It would be interesting to find a way to deduce the full Theorem \ref{thm:ERRSS} from this local analytic form.
Something close to that is done by Tao in \cite[Theorem 41]{T15} using arguments from complex analysis.

\paragraph{Distances from three points.}
The following result follows from Theorem \ref{thm:ESRSZ}.

\begin{theorem}[Elekes-Szab\'o, Sharir-Solymosi]\label{thm:ESSS}
Given three non-collinear points $p_1,p_2,p_3$ and a point set $P$ in $\R^2$,
there are $\Omega(|P|^{6/11})$ distinct distances from $p_1,p_2,p_3$ to $P$.
\end{theorem}
\begin{proof}[Proof sketch]
Let $D$ denote the set of squared distances between $p_1, p_2,p_3$ and the points in $P$.
A point $q\in P$ determines three squared distances to $p_1, p_2,p_3$, given by
$$
a  = (x_q-x_{p_1})^2 + (y_q-y_{p_1})^2,~~
b  = (x_q-x_{p_2})^2 + (y_q-y_{p_2})^2,~~
c  = (x_q-x_{p_3})^2 + (y_q-y_{p_3})^2 .
$$
The variables $x_q$ and $y_q$ can be eliminated from these equations to yield a quadratic equation $F(a,b,c)=0$
with coefficients depending on $p_1, p_2,p_3$ (in \cite{ES12} $F$ can be seen written out).
By construction, for each point $q\in P$, the corresponding squared distances $a,b,c$ belong to $D$. 
The resulting triples $(a,b,c)$ are all distinct, so $F$ vanishes at $|P|$ triples of $D\times D\times D$. 

The polynomial $F(x,y,z)$ turns out to be quadratic in $z$, so we can locally express it as $z = f(x,y)$.
Then we can apply Lemma \ref{lem:derivative} to $f$ to see that, if $p_1,p_2,p_3$ are not collinear, then $f$ does not have the form in \eqref{eq:testform}, 
which implies that $F$ does not satisfy property $(ii)$ of Theorem \ref{thm:ESRSZ}. 
Then property $(i)$ gives $|P| = O(|D|^{11/6})$, or $|D| = \Omega(|P|^{6/11})$.
When $p_1,p_2,p_3$ are collinear, $F$ becomes
a linear polynomial, so it does satisfy property $(ii)$.
\end{proof}

This problem was introduced by Elekes \cite{E95}, who showed that if $p_1,p_2,p_3$ are collinear (and equally spaced), then one can place $P$ so that there are only $O(|P|^{1/2})$ distances from $p_1, p_2,p_3$ to $P$.
Elekes and Szab\'o \cite{ES12} proved Theorem \ref{thm:ESSS} with the weaker bound $\Omega(|P|^{1/2+\eta})$ for some small absolute constant $\eta>0$, as a consequence of their version of Theorem \ref{thm:ESRSZ} (although they formulated the result in terms of triple points of circles; see the next application).
Sharir and Solymosi \cite{SS15} used the setup of \cite{SSS13} and ad hoc arguments to improve this $\eta$ to $1/22$; their work preceded \cite{RSZ15} and was the first extension of \cite{SSS13} that does not follow from \cite{RSS15a}.

Theorem \ref{thm:ESRSS} provides curious new information on Erd\H os's distinct distances problem (see \cite{BMP05}), which asks for the minimum number of distinct distances determined by a point set in $\R^2$.
Erd\H os conjectured that this minimum is $\Theta(n/\sqrt{\log n})$, and this was almost matched by Guth and Katz \cite{GK15}\footnote{The new algebraic methods introduced in \cite{GK15} indirectly led to the improvement of \cite{SSS13} in Theorem \ref{thm:ERSSS}, and thus to many of the recent results in this survey.}, who established $\Omega(n/\log n)$.
Theorem \ref{thm:ESSS} suggests that something stronger is true: There are many distances that occur just from three fixed non-collinear points.
If, as conjectured, the bound in Theorem \ref{thm:ESRSZ} can be improved from $O(n^{11/6})$ to $O(n^{1+\eps})$, or even $O(n\sqrt{\log n})$, 
then Erd\H os's conjectured bound would already hold if one only considers distances from three non-collinear points in the point set (note that if the entire point set is collinear, there are $\Omega(n)$ distances from any given point).

While waiting for improvements in the bound of Theorem \ref{thm:ESSS}, we could instead consider distances from more than three points.
Given $k$ points in a suitable non-degenerate configuration and a point set $P$ in $\R^2$, 
the number of distances from the $k$ points to $P$ should be $\Omega(|P|^{1/2+\alpha_k})$, where we would expect $\alpha_k$ to grow with $k$.
Let us pose the first unknown step as a problem.

\begin{problem}
Let $P\subset \R^2$, and consider four points in $\R^2$ such that no three are collinear (or such that some stronger condition holds).
Then the number of distinct distances from the four points to $P$ is $\Omega(|P|^{1/2 + \alpha})$, with $\alpha > 1/22$.
\end{problem}

\paragraph{Triple points of circle families.}
Elekes and Szab\'o \cite{ES12} formulated the problem of Theorem \ref{thm:ESSS} in a different way.
They considered three points $p_1,p_2,p_3$ in $\R^2$ and  three families of $n$ concentric circles centered at the three points,
and they looked for an upper bound on the number of \emph{triple points} of these families, i.e., points covered by one circle from each family.
Theorem \ref{thm:ESSS} states that if $p_1,p_2,p_3$ are not collinear, then the number of triple points is $O(n^{11/6})$; a construction in \cite{E95} shows that if $p_1,p_2,p_3$ are collinear, there can be as many as $\Omega(n^2)$ triple points.

One can ask the same question for any three one-dimensional families of circles, or even more general curves.
Such statements were studied by Elekes, Simonovits, and Szab\'o \cite{ESS09},
with a special interest in the case of concurrent unit circles, i.e., unit circles passing through a fixed point.
This case was improved by Raz, Sharir, and Solymosi \cite{RSS15b}, again using the setup of \cite{SSS13}, and ad hoc analytic arguments.
By the arguments from \cite{ESS09}, the improvement also follows from Theorem \ref{thm:ESRSZ}.

\begin{theorem}[Elekes-Simonovits-Szab\'o, Raz-Sharir-Solymosi]
\label{thm:ESRSS}
Three families of $n$ concurrent unit circles (concurrent at three distinct points) determine $O(n^{11/6})$ triple points.
\end{theorem}

This result shows an interesting distinction between lines and unit circles,
because for three families of concurrent lines it is possible to determine $\Omega(n^2)$ triple points.
We can for instance take the horizontal and vertical lines of an $n\times n$ integer grid, and $n$ lines at a $45^\circ$ angle that cover $\Omega(n^2)$ points of the grid.

It is natural to extend the question to ask for $k$-fold points of $k$ families of concurrent unit circles (or other curves).
This is largely unexplored when $k$ is small. 
For large $k<\sqrt{n}$ and an arbitrary set of $n$ unit circles, it follows from the incidence bound of Pach and Sharir \cite{PS98} (the general case of Theorem \ref{thm:SZ}) that there are at most $O(n^2/k^3)$ points where at least $k$ circles meet.

\subsection{About the proof of Theorem \ref{thm:ESRSZ}}
\label{subsec:ESproof}

Let us briefly discuss the proof of Theorem \ref{thm:ESRSZ}, although again we will not go into too much detail.
The proof is based on that of Theorem \ref{thm:ERRSS}, as described in Subsection \ref{subsec:ERproof},
but now we have to deal with $F(x,y,z)=0$ instead of $f(x,y) = z$.

The first challenge is that we can no longer define the quadruples and curves using the equation $f(a,b) = f(a',b')$.
Instead, we define the quadruples by
\[ Q = \{(a,b,a',b')\in A\times B\times A\times B:
\exists c\in C ~\text{such that}~ F(a,b,c) = F(a',b',c) =0\}.\]
Using the Cauchy-Schwarz inequality we get
\begin{align*}
 |Z(F)\cap (A\times B\times C)| & = \sum_{c\in C}|\{(a,b)\in A\times B : F(a,b,c)=0\}|\\
 & \leq |C|^{1/2}\left( \sum_{c\in C} |\{(a,b)\in A\times B : F(a,b,c)=0\}|^2\right)^{1/2}\\
 & =O_d\left(|C|^{1/2}|Q|^{1/2}\right).
 \end{align*}
 In the last step we use the fact that for $(a,b)\in A\times B$, there are at most $d$ values of $c\in \C$ for which $F(a,b,c)=0$ (unless $F(x,y,z)$ contains a vertical line, but this happens at most $O_d(1)$ times).
 Again the goal is to obtain the upper bound $|Q| = O_d(n^{8/3})$ using an incidence bound, which will result in $|Z(F)\cap (A\times B\times C)|  = O_d(n^{1/2}\cdot (n^{8/3})^{1/2}) = O_d(n^{11/6})$.
 
 The set $Q$ can be viewed as the \emph{projection} of a \emph{fiber product}.
 A fiber product\footnote{This is a special case of a more general object from category theory; what we call a fiber product here is sometimes called a \emph{set-theoretic} fiber product, or also a \emph{relative product}.} of a set with itself has the form $X\times_\varphi X = \{(x,x')\in X\times X : \varphi(x) = \varphi(x')\}$ for some function $\varphi:X\to Y$.
 This type of product is useful for counting, because the Cauchy-Schwarz inequality gives 
 $|X| \leq |X\times_\varphi X|^{1/2} |Y|^{1/2}$.
 In the calculation above, we have $X = |Z(F)\cap (A\times B\times C)|$, $Y = C$, and $\varphi(a,b,c) = c$.
 Then $Q$ is the projection of $X\times_\varphi X$ to the coordinates $(a,b,a',b')$, and has essentially the same size as $X\times_\varphi X$.
See \cite{BT12, T15} for similar uses of fiber products, as well as further discussion of the technique.
 
 The step in which we project from the fiber product to $Q$ is necessary to make the next step work; specifically, we need $Q$ to lie on a codimension one subvariety of $\C^4$, in order to be able to define the algebraic curves that we apply the incidence bound to.
 Unfortunately, this projection brings in the problem of \emph{quantifier elimination}.
The set $Q$ lies on the set 
\[\{(x,y,x',y')\in \C^4: \exists z\in \C ~\text{such that}~ F(x,y,z) = F(x',y',z)=0\},\]
which is not quite a variety, but only a \emph{constructible set} (see \cite{RSZ15} for details and references). 
We can eliminate the quantifier in the sense that there is a variety  $Z(G) \subset \C^4$ that contains the constructible set, and it differs only in a lower-dimensional set.
However, we have little grip on $G$ other than that its degree is bounded in terms of that of $F$.
Note that in the proof of Theorem \ref{thm:ERRSS} in Subsection \ref{subsec:ERproof} we had $F = f(x,y) - z$, so that we could easily eliminate $z$ to get $f(x,y ) = f(x',y')$.

To obtain $|Q| = O_d(n^{8/3})$, we define curves as in Subsection \ref{subsec:ERproof}, but this becomes more complicated due to the quantifier elimination.
We set 
\[ C_{aa'} = \{(x,y)\in \C^2: \exists z\in \C ~\text{such that}~ F(a,x,z) = F(a',y,z) =0\}.\]
This set is a one-dimensional constructible set, i.e., an  algebraic curve with finitely many points removed.
This leads to many technical complications, but we can basically still apply an incidence bound to the points and curves
\[ \pts = B\times B,~~~ \cvs = \{C_{aa'}:(a,a')\in A\times A\},\]
and we essentially have $|Q| = |I(\pts,\cvs)|$.
As in Subsection \ref{subsec:ERproof}, we can use Theorem \ref{thm:SZ} to obtain $|I(\pts,\cvs)| = O_d(n^{8/3})$, unless the curves badly violate the degrees-of-freedom condition.
The hardest part of the proof is then to connect the failure of the degrees-of-freedom condition to the special form $(ii)$ in Theorem \ref{thm:ESRSZ}.

\section{Elekes-R\'onyai problems on curves}
\label{sec:curves}

In this section we discuss some variants of the Elekes-R\'onyai and Elekes-Szab\'o problems for point sets contained in algebraic curves.
We still work with Cartesian products of ``one-dimensional'' finite sets, but instead of finite subsets of $\R$ or $\C$, we take finite subsets of algebraic curves.
We start with the first known instance of an 
Elekes-R\'onyai problem on curves, where the function is the Euclidean distance.
After that we discuss more general polynomial functions on curves, and finally we look at Elekes-Szab\'o problems on curves.

\subsection{Distances on curves}

We have already seen one instance of the Elekes-R\'onyai problem for distances on curves in Theorem \ref{thm:ERSSS}, which concerned distances between two point sets on two lines.
A natural generalization is to consider distances between two point sets on two algebraic curves in $\R^2$.
More precisely, given algebraic curves $C_1,C_2\subset \R^2$ and finite point sets $P_1\subset C_1, P_2\subset C_2$ with $|P_1|=|P_2| = n$, can we get a superlinear lower bound on $|D(P_1\times P_2)|$?\footnote{As in Subsection \ref{subsec:ERapps}, $D(p,q) = (p_x-q_x)^2+(p_y-q_y)^2$ is the squared Euclidean distance function.}

Observe that there are pairs of curves for which we cannot expect a superlinear lower bound on $|D(P_1\times P_2)|$, as we saw in Subsection \ref{subsec:ERapps} for parallel lines and orthogonal lines.
There is one more construction involving curves other than lines.
If we take two concentric circles with equally spaced points, then there is also only a linear number of distances between the two point sets.
It was proved by Pach and De Zeeuw \cite{PZ13} that these three constructions are the only exceptions to a superlinear lower bound; the proof used (once again) the setup in \cite{SSS13}, together with ad hoc arguments.

\begin{theorem}[Pach-De Zeeuw]\label{thm:PZ}
Let $C_1,C_2\subset \R^2$ be irreducible algebraic curves of degree at most $d$.
For finite subsets $P_1\subset C_1, P_2\subset C_2$ of size $n$ we have
\[|D(P_1\times P_2)| = \Omega_d(n^{4/3}), \]
unless the curves are parallel lines, orthogonal lines, or concentric circles.
\end{theorem}

The proof in \cite{PZ13} allows $C_1$ and $C_2$ to be the same curve, which leads to a statement for distances on a single curve that is interesting in its own right.
A version of this corollary was proved earlier by Charalambides \cite{C14}, but with a weaker bound $\Omega_d(n^{5/4})$ (the same exponent as in \cite{E99}, coming from the same counting scheme).
The proof in \cite{C14} relied on an interesting connection with \emph{graph rigidity}.

\begin{corollary}[Charalambides, Pach-De Zeeuw]\label{cor:CPZ}
Let $C\subset \R^2$ be an irreducible algebraic curve of degree $d$.
For a finite subset $P\subset C$ we have
\[|D(P\times P)| = \Omega_d(n^{4/3}), \]
unless $C$ is a line or a circle.
\end{corollary}

Let us discuss some of the new issues involved in the proofs of these results.
In the case of lines, a bound could be deduced from the Elekes-R\'onyai theorem (or proved directly), unless the lines are parallel or orthogonal.
For general algebraic curves, this does not seem possible. 
If the curves happen to be parametrized by polynomials, then plugging that parametrization into $D(p,q)$ would give a polynomial in two variables, and we could apply Theorem \ref{thm:ERRSS} (although it requires some work to translate the exceptional form of the polynomial to exceptional curves).
This was done in \cite{RSS15a} to prove that if $C$ is polynomially parametrizable, then the bound in Corollary \ref{cor:CPZ} holds, unless $C$ is a line (a circle is not polynomially parametrizable).

If the curves are not polynomially parametrizable, then this approach will not work.
The curve $y^2 = x^3+1$, for instance, has no parametrization with polynomials or rational functions.
We could locally write $y = \sqrt{x^3+1}$ and plug that into $D(p,q)$, but this gives an algebraic function in two variables; moreover, most curves do not even have such an explicit solution by radicals, and the best we can do is to use the implicit function theorem to locally write $y$ as an analytic function of $x$.
This suggests that Theorem \ref{thm:ESRSZ} may provide a larger framework for Theorem \ref{thm:PZ}.
We will see in Subsection \ref{subsec:ESoncurves} how one can manipulate a question about curves to fit it into the framework of Theorem \ref{thm:ESRSZ}.
Even then, extracting the exceptional forms of the curves from property $(ii)$ is not straightforward;
for distances on curves this was done by Raz and Sharir \cite{RS16}, who used considerations from graph rigidity (similar to those in \cite{C14}) to deduce Theorem \ref{thm:PZ} from Theorem \ref{thm:ESRSZ}.

A natural way to extend Theorem \ref{thm:PZ} or Corollary \ref{cor:CPZ} is to consider curves in higher dimensions.
This was done for Corollary \ref{cor:CPZ} by Charalambides \cite{C14} with the bound $\Omega(n^{5/4})$, using the counting scheme of \cite{E99}, together with various tools from analysis.
He determined that in $\R^3$ the exceptional curves are again lines and circles, 
while in $\R^4$ and above, the class of exceptional curves consists of so-called \emph{algebraic helices}; see \cite{C14,R16} for definitions and exact statements.
The bound in $\R^D$ was improved to $\Omega(n^{4/3})$ in \cite{RSS15a} for the case of polynomially parametrizable curves using Theorem \ref{thm:ERRSS},
and finally for all curves except algebraic helices by Raz \cite{R16}, using both Theorem \ref{thm:ESRSZ} and the analysis of Charalambides \cite{C14}.
Bronner, Sharir, and Sheffer \cite{BSS16} considered variants involving curves in $\R^D$ that are not necessarily algebraic.

Another variant of Theorem \ref{thm:PZ} was studied by Sheffer, Zahl, and De Zeeuw \cite{SZZ16}: Suppose that $P_1$ is contained in a curve $C$, and $P_2$ is an arbitrary set in $\R^2$. 
In general it is difficult to obtain bounds in this situation, 
but \cite{SZZ16} showed that if $C$ is a line or a circle, then the number of distances determined by $P_1\cup P_2$ is reasonably large. 
This result was used to show that a point set that determines $o(n)$ distinct distances cannot have too many points on a line or a circle, which is a small step towards the conjecture of Erd\H os that a set of $n$ points with $o(n)$ distinct distances must resemble an integer grid.

\subsection{Other polynomials on curves}

We can ask the same questions for any polynomial function $\R^2\times\R^2\to\R$, considered as a function $C_1\times C_2\to \R$.
For clarity, we focus on the case $C_1=C_2$. 
We also switch from $\R$ to $\C$, because some of the statements become more natural over $\C$.

Charalambides \cite{C14} also considered the function $A(p,q) = p_xq_y - p_yq_x$, which (in $\R^2$) gives twice the signed area of the triangle spanned by $p$, $q$, and the origin.
He proved that, for $P$ contained in an irreducible algebraic curve $C$ in $\R^2$, we have $|A(P\times P)| = \Omega_d(|P|^{5/4})$, unless $C$ is a line, an ellipse centered at the origin, or a hyperbola centered at the origin.
This result was generalized by Valculescu and De Zeeuw \cite{VZ15}, to any bilinear form over $C$, with a broader class of exceptional curves.
The condition on the curve is tight in the sense that for any excluded curve there is a bilinear form that can take a linear number of values on that curve.
To the authors, it was surprising that such curves can have large degree, whereas previous evidence (namely \cite{C14} and \cite{ES13}) suggested that the exceptional curves in such problems have degree at most three.

\begin{theorem}\label{thm:VZ}
Let $C$ be an irreducible algebraic curve in $\C^2$ of degree $d$,
and consider the bilinear form $B_A(p,q) = p^TAq$ for a nonsingular $2\times 2$ matrix $A$.
For $P\subset C$ we have
\[|B_A(P\times P)|=\Omega_d\left(|P|^{4/3}\right),\]
unless $C$ is a line, or linearly equivalent\footnote{We say that two curves are \emph{linearly equivalent} if there is a linear transformation $(x,y)\mapsto (ax+by,cx+dy)$ that gives a bijection between the point sets of the curves.} to a curve of the form $x^k = y^\ell$, with $k,\ell\in \Z\backslash\{0\}$.
\end{theorem}

The description of the exceptional curves in Theorem \ref{thm:VZ} is very succinct but requires some clarification.
Lines through the origin are linearly equivalent to  a curve of the form $x^k=y^\ell$, but other lines are not, which is why they are listed separately.
When $k$ or $\ell$ is negative, one obtains a more natural polynomial equation after multiplying by an appropriate monomial.
Thus hyperbola-like curves of the form $x^ky^\ell=1$ with coprime $k,\ell\geq 1$ are included, since they can also be defined by $x^k = y^{-\ell}$. 
Ellipses centered at the origin are also included, since these are linearly equivalent to the unit circle $(x-iy)(x+iy)=1$, which is linearly equivalent to $xy=1$. Thus all the exceptional curves of Charalambides are special.
Note that these curves are all \emph{rational}, i.e., they have a parametrization by rational functions.

The reason that these curves are exceptional in the proof of Theorem \ref{thm:VZ} is that they have infinitely many \emph{linear automorphisms}.
Here an \emph{automorphism} of a curve $C$ is a map $T:\C^2\to\C^2$ such that $T(C) = C$, and it is \emph{linear} if it is a linear transformation (and similarly one can define affine, projective, or rational automorphisms).
It was proved in \cite{VZ15} that the algebraic curves with infinitely many linear automorphisms are exactly those excluded in Theorem \ref{thm:VZ}.
From the proof of Theorem \ref{thm:VZ} in \cite{VZ15}, it appears that if one considers more general polynomial functions instead of $B_A$, 
the exceptional curves will be those that have infinitely many rational automorphisms.
By a theorem of Hurwitz (see for instance \cite[Exercise IV.2.5]{Ha77}), a nonsingular curve with infinitely many rational automorphisms must have genus zero or one, or in other words, it must be rational or elliptic.

If we consider the statement of Theorem \ref{thm:VZ} for more general polynomials, then we also encounter exceptional polynomials that can take a linear number of values on \emph{any} curve.
Already for the bilinear forms $B_A$, this occurs when $A$ is a singular matrix; in that case we can write $B_A(p,q) = L_1(p)\cdot L_2(q)$ with linear polynomials $L_1,L_2$, which is reminiscent of the multiplicative form in Theorem \ref{thm:ERRSS}.
More generally, for functions of the form $G(H(p)+K(q))$ or $G(H(p)\cdot K(q))$ there are exceptional constructions.
We arrive at the following conjecture (where again we could conjecture a larger exponent).

\begin{conjecture}\label{conj:ERoncurves}
Let $C\subset \C^2$ be an algebraic curve of degree at most $d$ and $F:C\times C\to \C$ a polynomial of degree  at most $d$.
Then for any $P\subset C$ we have
$$|F(P\times P)|=\Omega_d(|P|^{4/3}),$$
unless $F(p,q) = G(H(p)+K(q))$ or $F(p,q)=G(H(p)\cdot K(q))$,
or unless $C$ is rational.
\end{conjecture}

\subsection{Elekes-Szab\'o problems on curves}
\label{subsec:ESoncurves}

An Elekes-Szab\'o theorem on curves would take the following form.
Let $\cv_1, \cv_2, \cv_3\subset \C^2$ be algebraic curves of degree at most $d$, and let $G\in \C[x,y,s,t,u,v]$ be a polynomial of degree at most $d$.
Then for point sets $P_1\subset\cv_1,P_2\subset\cv_2,P_3\subset\cv_3$ of size $n$, we would want to bound $|Z(G)\cap (P_1\times P_2\times P_3)|$.
We would expect exceptions for certain $G$, of a form related to that in Theorem \ref{thm:ESRSZ}$(ii)$,
and we would expect exceptions for certain curves, including low-degree curves and those in Conjecture \ref{conj:ERoncurves}.
We won't state a full conjecture here, but we will discuss some of the instances that have been considered.

One possible choice of polynomial is
\[G(x,y,s,t,u,v) = \frac{1}{2}\begin{vmatrix} x&s&u\\y&t&v\\1&1&1\end{vmatrix} - 1,\]
for which $G(x,y,s,t,u,v)=0$ if and only if the triangle determined by the points $(x,y),(s,t),(u,v)$ has unit area.
The problem of determining the maximum number of unit area triangles determined by $n$ points in $\R^2$ is ascribed to Oppenheim in \cite{EP71}. 
The best known lower bound is $\Omega(n^2\log\log n)$, due to Erd\H os and Purdy \cite{EP71}, and the best known lower bound is $O(n^{20/9})$, due to Raz and Sharir \cite{RS15a}.
It is easy to show that for $n$ points on a curve, or three sets of $n$ points on three curves, 
the bound $O(n^2)$ holds.
The following problem was suggested by Solymosi and Sharir \cite{SS15}.

\begin{problem}
Given $n$ points on an algebraic curve $\cv$ in $\R^2$, 
prove a bound better than $O(n^2)$ on the number of unit area triangles, 
or show that there are point sets on $\cv$ with $\Omega(n^2)$ unit area triangles.
\end{problem}

The only case of this problem that has been studied is the one where $\cv$ is the union of three distinct lines.
Surprisingly, Raz and Sharir \cite{RS15a} showed that on any three distinct lines there are non-trivial constructions that determine $\Omega(n^2)$ unit area triangles.
They discovered this using the derivative criterion in Lemma \ref{lem:derivative}.

Another choice of $G$ (which we won't try to write out) is the polynomial such that $G(x,y,s,t,u,v)=0$ if the points $(x,y),(s,t),(u,v)$ lie on a common unit circle, or in other words, the circle determined by the three points has radius equal to one.
Since the radius of the circle determined by three points can be written as a rational function, multiplying out denominators gives a polynomial $G$ with the property above.
Raz, Sharir, and Solymosi showed that for three distinct unit circles in $\R^2$ with three finite point sets of size $n$, 
the number of determined unit circles is $O(n^{11/6})$ (this statement is equivalent to Theorem \ref{thm:ESRSS}).
One can generalize this problem as follows.

\begin{problem}
Given three algebraic curves $\cv_1, \cv_2, \cv_3\subset \R^2$ containing three sets of $n$ points,
prove a bound better than $O(n^2)$ 
on the number of unit circles containing one point from each of the three sets.
Are there curves for which $\Omega(n^2)$ unit circles is possible?
\end{problem}

Another problem of this form is to bound \emph{collinear triples} on curves. 
More is known on this problem, and we will discuss it in detail in the next subsection.
Note that it can be seen as a degenerate case of the unit area triangle problem, where instead of unit area we ask for \emph{zero} area.

\bigskip

Let us see how the Elekes-Szab\'o problem on curves is related to the Elekes-Szab\'o problem in $\C^3$.
We can think of the three curves $\cv_1, \cv_2, \cv_3\subset \C^2$ as a Cartesian product $\cv_1\times \cv_2\times \cv_3\subset \C^6$.
We can choose a generic projection $\varphi:\C^2\to\C$, in such a way that the three-fold product $\pi=(\varphi\times\varphi\times\varphi):\C^6\to \C^3$ maps a product $P_1\times P_2\times P_3\subset \cv_1\times \cv_2\times \cv_3$ with $|P_1|=|P_2|=|P_3|=n$ to a product $\varphi(P_1)\times \varphi(P_2)\times \varphi(P_3)\subset \C^3$ with $\varphi(P_1) = \varphi(P_2) = \varphi(P_3) = n$.

The variety $X = Z(F)\cap (\cv_1\times \cv_2\times \cv_3)$ is two-dimensional (unless $G$ happens to vanish on $\cv_1\times \cv_2\times \cv_3$, in which case the problem is trivial).
Again by choosing $\varphi$ generically, we get that $\pi(X)\subset \C^3$ is also a two-dimensional variety, and thus can be written as $X = Z(F)$.
Now we have
\[|Z(G)\cap (P_1\times P_2\times P_3)| \leq  |Z(F)\cap (\varphi(P_1)\times \varphi(P_2)\times \varphi(P_3))|,\]
so if the upper bound of Theorem \ref{thm:ESRSZ}$(i)$ applies to $F$, then we also have that upper bound for $G$ on $\cv_1\times \cv_2\times \cv_3$.
Otherwise, $F$ satisfies property $(ii)$ of Theorem \ref{thm:ESRSZ}.

Unfortunately, it is not clear how to transfer back property $(ii)$ for $F$ to the original setting.
Note that $F$ not only encodes information about $G$, but also about the curves $\cv_1,\cv_2,\cv_3$.
Thus property $(ii)$ for $F$ should imply that either $G$ has an exceptional form, or the curves $\cv_1,\cv_2,\cv_3$ have an exceptional form.
This was made to work for the problem of collinear triples (see Subsection \ref{subsec:collineartriples}), but it remains difficult to do this in general.

\bigskip

The projection from $\C^6$ to $\C^3$ above lets us connect Theorem \ref{thm:ESRSZ} to the group law on elliptic curves (irreducible nonsingular algebraic curves of degree three).
We refer to \cite{ST92} for an introduction to the group structure of an elliptic curve, but we summarize it here as follows: 
On any elliptic curve $E$ there are an operation $\oplus$ and an identity element $\mathcal{O}$  that turn the point set $E$ into a group, with the special property that $P,Q,R\in E$ are collinear if and only if $P\oplus Q\oplus R = \mathcal{O}$.

This group structure on elliptic curves allows us to construct finite point sets with many collinear triples.
If we take a finite subgroup $H$ of size $n$ of an elliptic curve $E$ (which exists for any $n$), then $|Z(G)\cap (H\times H\times H)| = \Omega(n^2)$, where $G$ is the polynomial that represents collinearity.
This follows from the fact that for any two distinct elements $P,Q\in H$, the line spanned by $P$ and $Q$ intersects $E$ in $R = \ominus (P\oplus Q)$ (the group element $R$ such that $(P\oplus Q)\oplus R = \mathcal{O}$), which must also be in the subgroup $H$.
We could have $R = P$ or $R = Q$, but there are only $O(n)$ pairs that satisfy $P\oplus P\oplus Q = \mathcal{O}$, so we get $\Omega(n^2)$ collinear triples of distinct points.

Applying the generic projection $\pi:\C^6\to \C^3$, we get a polynomial $F\in \C[x,y,z]$ such that $|Z(F)\cap (\varphi(H)\times \varphi(H)\times \varphi(H))|=\Omega(n^2)$, which implies that $G$ satisfies property $(ii)$ of Theorem \ref{thm:ESRSZ}.
It follows that, locally, there are $\varphi_i:\C\to \C$ such that $P,Q,R\in E$ are collinear if and only if $\varphi_1(\varphi(P)) + \varphi_2(\varphi(Q)) + \varphi_3(\varphi(R)) = 0$. 
This is why it appears that the local analytic nature of property $(ii)$ is necessary; it would be a big surprise if such maps existed globally, or could be described by polynomial or rational functions.

\subsection{Collinear triples on curves}
\label{subsec:collineartriples}

The one instance of an Elekes-Szab\'o problem on curves that has been studied in some detail (and is not an Elekes-R\'onyai problem) is the problem where $F$ represents collinearity.
This question was first considered by Elekes and Szab\'o \cite{ES13}; they proved a weaker bound in $\R^2$ using the main bound from \cite{ES12}, which was then improved in \cite{RSZ15} to the following statement.

\begin{theorem}[Elekes-Szab\'o, Raz-Sharir-De Zeeuw]\label{thm:ESRSZcoltrip}
Let $\cv$ be an irreducible algebraic curve in $\C^2$ of degree $d$ and let $P\subset \cv$ be a finite set.
Then $P$ determines $O_d(n^{11/6})$ proper collinear triples, unless $\cv$ is a line or a cubic curve.
\end{theorem}

As we saw at the end of Subsection \ref{subsec:ESoncurves}, elliptic curves must be exceptions to this statement, because their group structure gives constructions with a quadratic number of collinear triples.
In fact, on any cubic curve (including the union of a conic and a line, or a union of three lines) there is a ``quasi-group law''; see Green and Tao \cite[Proposition 7.3]{GT13}. 
This law does not quite give the whole point set a group structure (and there may not be an identity element), but comes close enough to allow for a construction with $\Omega(n^2)$ collinear lines.
These constructions are worked out in \cite{ES13}.

Green and Tao \cite{GT13} proved the Dirac-Motzkin conjecture for large point sets, which states that any non-collinear point set $P\subset\R^2$ of size $n$ determines at least $n/2$ \emph{ordinary lines} (lines containing exactly two points of $P$).
As a by-product, they also solved Sylvester's ``orchard problem'' for large $n$, which asks for the maximum number of lines with at least three points (\emph{triple lines}) from a set of $n$ points.
It is easy to see that this number is at most $\frac{1}{3}\binom{n}{2}$, and Sylvester noted (in 1868; see \cite{GT13}) that there are constructions on elliptic curves with almost exactly this number of triple lines.
Green and Tao showed that any set with at least $\frac{1}{6}n^2 - O(n)$ triple lines must have most of its points on a cubic curve.
Elekes made the bolder conjecture that any set with $\Omega(n^2)$ collinear triples has many points on a cubic;
to underline our ignorance he suggested to show that ten of the points are on a cubic (which is one more than the trivial number).
The conjecture was stated in \cite{ES13} but already present in disguise in \cite{E02}.

\begin{conjecture}[Elekes]
If $P\subset \R^2$ determines $\Omega(n^2)$ collinear triples, then at least ten points of $P$ lie on a cubic.
\end{conjecture}

An interesting application of Theorem \ref{thm:ESRSZcoltrip}, pointed out in \cite{ES13}, is to the problem of \emph{distinct directions} mentioned in Subsection \ref{subsec:ERapps}.
Theorem \ref{thm:ESRSS} allows us to generalize Corollary \ref{cor:directionsongraphs} to any algebraic curve.
The connection with triple lines is that if two pairs of points determine lines in the same direction, then these lines intersect the line at infinity in the same point.
Thus a point set on a curve $\cv$ that has few directions determines few such points at infinity, or conversely, adding these points at infinity gives a point set with many collinear triples.
An unbalanced form of Theorem \ref{thm:ESRSZcoltrip} then gives the following lower bound on the number of directions.
There is an exception when $\cv$ together with the line at infinity is a cubic curve, which means that $\cv$ is a conic.

\begin{corollary}
Let $\cv$ be an irreducible algebraic curve in $\C^2$ of degree $d$ and let $P\subset \cv$ be a finite set.
Then $P$ determines $\Omega_d(n^{4/3})$ distinct directions, 
unless $\cv$ is a conic.
\end{corollary}

\section{Other variants}\label{sec:variants}

\subsection{Longer one-dimensional products}
\label{subsec:longer}

One way to extend Theorems \ref{thm:ERRSS} and \ref{thm:ESRSZ} is to consider ``longer'' Cartesian products, i.e., products with more factors. 
The reference point for such bounds is the Schwartz-Zippel lemma,
which we now state in full generality.
This version was proved by Lang and Weil \cite[Lemma 1]{LW54}\footnote{This brings into question if ``Schwartz-Zippel'' is the right name, but it has become standard in combinatorics.}; see also Tao \cite{T12} for a proof sketch (both focus on finite fields, but their proofs work also over $\C$).
The bound in fact holds over any field, and can be modified to allow factors $A_i$ of different sizes.

\begin{theorem}\label{thm:ScZi}
Let $X\subset \C^D$ be a variety.
Then for finite sets $A_1,\ldots, A_D\subset \C$ of size $n$ we have
\[|X\cap (A_1\times \cdots \times A_D)| = O_{D,\deg(X)}(n^{\dim(X)}) .\]
\end{theorem}

Theorem \ref{thm:ESRSZ} told us that in the case with $D=3$ and $\dim(X) = 2$, this bound can be improved on, unless the defining polynomial of $X$ is special.
Analogously, one would expect that the bound in Theorem \ref{thm:ScZi} can be improved on, unless the variety is of some special type.
The only other case that has so far been studied is $D=4$ and $\dim(X)=3$, for which Raz, Sharir, and De Zeeuw \cite{RSZ16} proved the following.

\begin{theorem}[Raz-Sharir-De Zeeuw]\label{thm:ES4D}
Let $F\in \C[x,y,s,t]$ be an irreducible polynomial of degree $d$ with 
with each of $F_x,F_y,F_s,F_t$ not identically zero. 
Then one of the following holds.\\
$(i)$ For all $A,B,C,D\subset \C$ of size $n$ we have
$$|Z(F) \cap (A\times B\times C\times D)|=O_d(n^{8/3}).$$
$(ii)$ There exists a one-dimensional subvariety $Z_0\subset Z(F)$, such that every $v\in Z(F)\backslash Z_0$ has an open neighborhood $D_1\times D_2\times D_3\times D_4$ and analytic functions 
$\varphi_i: D_i\to \C$,
such that for every $(x,y,s,t)\in D_1\times D_2\times D_3\times D_4$ we have
$$(x,y,s,t)\in Z(F)~~~\text{if and only if}~~~\varphi_1(x)+\varphi_2(y)+\varphi_3(s)+\varphi_4(t) = 0.$$
\end{theorem}

The proof of Theorem \ref{thm:ES4D} mostly uses the same techniques as that of Theorem \ref{thm:ESRSZ}, and the setup turns out to be simpler.
The reason is that (because four is an even number) we can define curves by 
\begin{align}\label{eq:curvesfrom4var}
 C_{cd} = \{(x,y): F(x,y,c,d)=0\},
 \end{align}
which avoids quantifier elimination and thus many technical complications.
The theorem is in fact closely related to the incidence bound in Theorem \ref{thm:SZ}: In most applications of that bound, the curves are defined as in \eqref{eq:curvesfrom4var}, for some given polynomial $F$ 
(for instance, the proof of Theorem \ref{thm:ERRSS} in Subsection \ref{subsec:ERproof} uses $F=f(s,x) - f(t,y)$).
Theorem \ref{thm:ES4D} thus gives the same bound as Theorem \ref{thm:SZ}, but it replaces the combinatorial condition (two points being contained in a bounded number of curves) by an algebraic condition (that $F$ is not of the form in $(ii)$).
Much like condition $(ii)$ in Theorem \ref{thm:ESRSZ}, condition $(ii)$ in Theorem \ref{thm:ES4D} can often be checked using derivatives.

A consequence of Theorem \ref{thm:ES4D} is an expansion bound for three-variable polynomials $f\in \C[x,y,z]$ on sets $A,B,C\subset \C$ with $|A|=|B|=|C|=n$, namely
\[ |f(A\times B\times C)| = \Omega(n^{3/2}),\]
unless $f$ is in a local sense of the form $\psi(\varphi_1(x)+\varphi_2(y)+\varphi_3(s))$ with analytic $\psi, \varphi_1,\varphi_2,\varphi_3$.
A weaker bound for this question was proved in \cite{SSZ13}, with the more precise condition that $f$ is not of the form $g(h(x)+k(y)+l(z))$ or $g(h(x)\cdot k(y)\cdot l(z))$ with $g,h,k,l$ polynomials.
It should be possible to use techniques similar to those in \cite{RSS15a} to prove the stronger bound with the precise condition.

We make the following conjecture for the general case.

\begin{conjecture}
There is a constant $c>0$ such that one of the following holds for any irreducible $F\in \C[x_1,\ldots,x_D]$ of degree $d$ with each of $F_{x_1},\ldots,F_{x_D}$ not identically zero. \\
$(i)$ For all $A_1,\ldots,A_D\subset \C$ of size $n$ we have
$$|Z(F) \cap (A_1\times\cdots\times A_D)|=O_d(n^{D-1-c}).$$
$(ii)$ In a local sense we have
$$(x_1,\ldots,x_D)\in Z(F)~~~\text{if and only if}~~~\sum_{i=1}^D\varphi_i(x_i) = 0.$$
\end{conjecture}

A largely unexplored question is how the bound of Theorem \ref{thm:SZ} can be improved for a variety $X\subset \C^D$ with $\dim(X)<D-1$.
When $\dim(X) = 1$, the bound $O(n)$ is tight: We can place $n$ points on $X$ and project in the coordinate directions to get $A_i$ such that $|X\cap (A_1\times\cdots \times A_D)| = n$.
Thus the first open case is $D=4$, $\dim(X) = 2$.

\begin{problem}
Determine for which two-dimensional varieties $X\subset \C^4$ the Schwartz-Zippel-type bound $|X\cap (A_1\times A_2\times A_3 \times A_4)| = O_{\deg(X)}(n^2)$ is tight.
\end{problem}

\subsection{Two-dimensional products}

A different way of extending Theorems \ref{thm:ERRSS} and \ref{thm:ESRSZ} would be to consider Cartesian products of ``two-dimensional'' sets instead of ``one-dimensional'' sets, i.e., finite subsets of $\C^2$ instead of finite subsets of $\C$.
For instance, the two-dimensional analogue of the Elekes-R\'onyai problem would be: Given a map $\mathcal{F}:\C^2\times \C^2\to \C^2$, defined by $\mathcal{F}(x,y,s,t) = (F_1(x,y,s,t),F_2(x,y,s,t))$  for polynomials $F_1,F_2\in\C[x,y,s,t]$, 
and given finite sets $P,Q\subset \C^2$, can we obtain a non-trivial lower bound on $|\mathcal{F}(P\times Q)|$?

The study of such questions was initiated by Nassajian Mojarrad et al. \cite{NPVZ15}.
Even analogues of the Schwartz-Zippel lemma become more complicated: The easiest case would be a bound of the form $|X\cap (P\times Q)|$ for a variety $X\subset \C^4$ and finite sets $P,Q\subset \C^2$, but there are varieties for which no non-trivial bound holds.
Let us call a polynomial $F\in \C[x,y,s,t]$ \emph{Cartesian} if it can be written as
\begin{align}\label{eq:cartesian}
F(x,y,s,t) = G(x,y)H(x,y,s,t)+ K(s,t)L(x,y,s,t),
\end{align}
with $G\in \C[x,y]\backslash\C$, $K\in \C[s,t]\backslash\C$, and $H,L\in \C[x,y,s,t] $.
We say that a variety $X\subset \C^4$ is \emph{Cartesian} if every polynomial vanishing on $X$ is Cartesian with the same $G,K$.
For a Cartesian variety $X$, we can have $|X\cap (P\times Q)| = |P||Q|$, since we can take $P\subset Z(G)$ and $Q\subset Z(K)$ to get $P\times Q\subset X$.

It is proved in \cite{NPVZ15} that if $X$ is not Cartesian, then one can obtain non-trivial upper bounds on $|X\cap (P\times Q)|$.
As a consequence, we obtain the following lower bound on the number of distinct values of a polynomial map $\mathcal{F}=(F_1,F_2):\C^2\times\C^2\to \C^2$.

\begin{theorem}[Nassajian Mojarrad et al.]\label{thm:NPVZ}
Let $F_1,F_2\in \C[x,y,s,t]$ be polynomials of degree $d$ and $\mathcal{F}=(F_1,F_2)$.
Then for $P,Q\subset \C^2$ with $|P| = |Q|=n$ we have 
\[|\mathcal{F}(P\times Q)| = \Omega_d(n),\]
unless $F_1 = \varphi_1\circ \psi$ and $F_2 = \varphi_2\circ\psi$ for a nonlinear polynomial $\psi$, or $F_1$ and $F_2$ are both Cartesian with the same $G,K$.
\end{theorem}

This theorem provides a starting point for the study of Elekes-R\'onyai problems for two-dimensional polynomial maps.
To compare with the one-dimensional case, for any $f\in \C[x,y]\backslash \C$ and $A,B\subset \C$ of size $n$, we can deduce $|f(A\times B)|=\Omega(n)$ from the Schwartz-Zippel lemma (Theorem \ref{thm:ScZi}),
and Theorem \ref{thm:ERRSS} improves on that bound for polynomials that are not additive or multiplicative.
The question is thus for which polynomial maps (other than those already excluded in Theorem \ref{thm:NPVZ}) the bound $\Omega(n)$ can be improved on.

The bound of Theorem \ref{thm:NPVZ} is tight for certain polynomial maps.
For instance, if we take the vector addition map $\mathcal{F} = (x+s,y+t)$ and we take $P$ to be any arithmetic progression on a line, 
then we have $|\mathcal{F}(P\times P)| = O(n)$; note that $x+s$ and $y+t$ are both Cartesian, but not with the same $G, K$.
We can do something similar for any map of the form $\mathcal{F} = (f_1(x,s), f_2(y,t))$, where $f_1,f_2$ are additive or multiplicative in the sense of Theorem \ref{thm:ERRSS}.
Just like for the Elekes-R\'onyai problem, composing these maps with other polynomials gives more exceptions: If we write $(x,y)\oplus (s,t) = (x+s,y+t)$, and we set $\mathcal{F} = \psi(\varphi_1(x,y)\oplus\varphi_2(s,t))$ with reasonable maps $\psi,\varphi_1,\varphi_2:\C^2\to \C^2$, then again the image of $\mathcal{F}$ can have linear size.

\begin{problem}\label{prob:polymap}
Determine for which polynomial (or even rational) maps $\mathcal{F}:\C^2\times \C^2\to \C^2$
there exists an $\alpha>0$ such that
\[|\mathcal{F}(P\times Q)| = \Omega_d(n^{1+\alpha})\]
for all $P,Q\subset \C^2$ of size $n$, perhaps with further restrictions on $P$ and $Q$.
\end{problem}

Superlinear bounds like in Problem \ref{prob:polymap} are known for only a few functions (which happen to be rational maps), and only over $\R$.
Beck's ``two extremities'' theorem \cite[Theorem 3.1]{B83} can be phrased in terms of the rational map $L:\R^2\times\R^2\to \R^2$ that maps a pair of points to the line spanned by the pair (for instance represented by the point in $\R^2$ whose coordinates are the slope and intercept of the line).
Beck's theorem says that $|L(P\times P)| = \Omega_c(|P|^2)$, unless $P$ has at least $c|P|$ points on a line.
Raz and Sharir \cite{RS15a} work with the map that sends a pair of points to the line that consists of all the points that span a unit area triangle (of a fixed orientation) with the pair.
They prove a superlinear bound on the number of distinct values of this map, for point sets with not too many points on a line (this statement is not made explicit in the paper, but is implicit in the proof).
Finally, Lund, Sheffer, and De Zeeuw \cite{LSZ14} study the rational map $\mathcal{B}:\R^2\times\R^2\to \R^2$ that sends a pair of points to the line that is their perpendicular bisector.
They prove that $|\mathcal{B}(P\times P)| = \Omega_{M,\eps}(|P|^{8/5-\eps})$ if $P\subset \R^2$ has at most $M$ points on a line or circle.

Elekes and Szab\'o \cite{ES12} proved a more general theorem in this vein (their ``Main Theorem'' \cite[Theorem 27]{ES12}).
It is difficult even to state this theorem precisely, so we give only a rough description.
Let $Y$ be a $D$-dimensional irreducible variety over $\C$ of degree at most $d$, and let $X\subset Y\times Y\times Y$ be an irreducible $2D$-dimensional subvariety of degree at most $d$ with surjective and generically finite projections onto any two of the three factors.
Let $A\subset Y$ be a finite set, which is in very general position in the following sense:
For any proper subvariety $Z\subset Y$ of degree at most $M$ we have $|A\cap Z|\leq N$.
Then 
\[|X\cap (A\times A\times A)| = O_{d,D,M,N}(|A|^{2-\eta})\]
with $\eta>0$ depending on $d,D,M,N$,
unless $X$ is in some specific way related to an algebraic group.

The results of \cite{B83, RS15a, LSZ14} mentioned above fit into this framework: Given a reasonable map $\mathcal{F}:\C^2\times \C^2\to\C^2$, its graph $X \subset \C^2\times\C^2\times\C^2$ is a variety that satisfies the conditions of the Main Theorem of Elekes and Szab\'o.
Note that \cite{B83, RS15a, LSZ14} not only provide explicit values of $\eta$ in certain special cases, but they also make the ``very general position'' condition more precise in those cases. 
Indeed, they replace the condition that $P$ avoids any algebraic curve (a proper subvariety of $\C^2$) of bounded degree, 
with the condition that $P$ avoids only lines in the case of \cite{B83, RS15a}, or lines and circles in \cite{LSZ14}.

\subsection{Expanding polynomials over other fields}

Let us finish by briefly discussing Elekes-R\'onyai-type questions over a finite field $\mathbb{F}_q$ and the field $\mathbb{Q}$ of rational numbers.
We focus on statements of the form ``$|f(A\times A)|=\Omega(|A|^{1+c})$ for a polynomial $f$ that is not of a special form'';
there has been much work on conditional expansion bounds (especially over finite fields), of the form ``if $|A+A|$ is small, then $|f(A\times A)|$ is large'', but we will not discuss these here.

Over finite fields, the question is considerably harder, because there is not yet any finite field equivalent of an incidence bound for algebraic curves like Theorem \ref{thm:SZ}.
For simplicity, let us restrict to a prime field $\mathbb{F}_p$, and let us ignore small $p$.
There is a dichotomy between \emph{small} subsets of $\mathbb{F}_p$, 
where very few incidence bounds are known,
and \emph{large} subsets of $\mathbb{F}_p$, for which various techniques can be used to obtain incidence bounds.

For instance, Bourgain \cite{B05} used the incidence bound of Bourgain, Katz, and Tao \cite{BGK04} for small subsets of $\mathbb{F}_p$ to prove the following expansion bound.
For $f(x,y) = x^2+xy$ and $A\subset \F_p$ with $|A|<p^{c'}$, we have $|f(A\times A)| >|A|^{1+c}$, with $c>0$ depending on $c'<1$.
This bound has been improved slightly, and generalized to some other polynomials, but the range of polynomials for which such bounds are known remains limited; see Aksoy Yazici et al. \cite{AMRS15} for some recent developments.

For large subsets of finite fields (think of $|A|>p^{7/8}$), 
more comprehensive bounds have been proved.
Bukh and Tsimerman \cite{BT12} proved that $|f(A\times A)| = \Omega_d(|A|^{1+c})$ for $|A|>p^{7/8+c'}$, with $c>0$ depending on $c'$, 
if $f\in \mathbb{F}_p[x,y]$ is monic in each variable, not of the form $p(q(x,y))$ with $\deg p\geq 2$, and not of the form $g(x)+h(y)$ or $g(x)h(y)$.
The exceptional cases here are close to those in  Theorem \ref{thm:ERRSS}.
Tao \cite{T15} then proved that $|f(A\times A)| = \Omega_d(p)$ for $|A|>p^{15/16}$, 
unless $f\in \mathbb{F}_p[x,y]$ is additive or multiplicative, matching the condition in Theorem \ref{thm:ERRSS}.
The proofs in \cite{BT12,T15} both used fiber products and Cauchy-Schwarz (somewhat like in Subsection \ref{subsec:ESproof}),
and both relied on the bound of Lang and Weil \cite[Theorem 1]{LW54} for points on varieties over finite fields.
Some of the ideas in \cite{T15} played a role in the proof of Theorem \ref{thm:ESRSZ} in \cite{RSZ15}.
In \cite[Section 9]{BT12}, an Elekes-Szab\'o-type statement over finite fields is conjectured.

\bigskip

One can also ask the Elekes-R\'onyai question over $\Q$.
Of course, Theorem \ref{thm:ERRSS} gives a bound there, but it is not clear that every exceptional polynomial from Theorem \ref{thm:ERRSS} is also exceptional over $\Q$.
For instance, although $f(x,y)=x^2+y^2$ is additive, the construction from Subsection \ref{subsec:exppoly} would require us to choose $A$ so that $A^2$ is an arithmetic progression, which is not possible in $\Q$.
Solymosi made the following conjecture.

\begin{conjecture}
Let $f\in \Q[x,y]$ be a polynomial of degree $d$.
For $A\subset \Q$ we have
\[|f(A\times A)| = \Omega_d(|A|^{1+c}),\]
unless $f(x,y) = g(ax+by)$ or $f(x,y) = g((x+a)^\alpha(y+b)^\beta)$ for $a,b\in \Q$ and positive integers $\alpha, \beta$.
\end{conjecture}


\end{document}